\documentclass[a4paper]{amsart}
\usepackage[utf8]{inputenc}
\usepackage{eucal}
\usepackage{tikz-cd}
\usepackage{amssymb,amsmath,amsthm}
\usepackage{lmodern}
\usepackage[T1]{fontenc}
\usepackage[hidelinks]{hyperref}
\usepackage{mathrsfs}
\usepackage{mathtools}
\usepackage{graphicx}
\usepackage{tikz}
\usepackage{dsfont} %for the identity matrix (\mathbb{1} clashes with ams?)

\calclayout

\usepackage{thmtools}
\declaretheorem[name=Theorem, parent=section]{thm}
\declaretheorem[name=Definition, sibling=thm]{defn}
\declaretheorem[name=Proposition, sibling=thm]{prop}
\declaretheorem[name=Corollary, sibling=thm]{cor}
\declaretheorem[name=Lemma, sibling=thm]{lem}
\declaretheorem[name=Example, style=definition, sibling=thm]{example}
\declaretheorem[name=Remark, style=definition, numbered=no]{rem}

\newcommand{\mf}{\mathfrak}
\newcommand{\mc}{\mathcal}
\newcommand{\ms}{\mathsf}
\newcommand{\on}{\operatorname}

\newcommand{\la}{\langle}
\newcommand{\ra}{\rangle}
\newcommand{\R}{\mathbb{R}}

\newcommand{\Z}{\mathbb{Z}}

\newcommand{\w}[1]{\wedge^{\!#1}\,}
\renewcommand{\dh}{\hat d\hspace{-.5mm}}

\usetikzlibrary{matrix,arrows,decorations.pathmorphing,decorations.markings,calc,patterns}
\tikzset{->-/.style={decoration={
  markings,
  mark=at position #1 with {\arrow{stealth'}}},postaction={decorate}}}
\tikzset{proj_empty/.style={circle,fill=white,inner sep=0pt}}
\tikzset{proj0/.style={circle,fill=white,draw=black,inner sep=0pt}}
\tikzset{proj1/.style={circle,fill=black,draw=black,inner sep=0pt}}
\tikzset{proj2/.style={circle,fill=blue,draw=blue,inner sep=0pt}}
\tikzset{proj3/.style={circle,fill=cyan,draw=cyan,inner sep=0pt}}

\title[G-algebroids]{G-algebroids: a unified framework for exceptional and generalised geometry, and Poisson--Lie duality}

\author{Mark Bugden}
\address{Mathematical Institute, Faculty of Mathematics and Physics\\Charles University Prague 186 75, Czech Republic}
\email{mathphys@mark-bugden.com}
\author{Ond\v rej Hul\' ik}
\address{Institute of Physics of the Czech Academy of Sciences, CEICO, Na Slovance 2, 182 21 Prague 8, Czech Republic
/
Institute of Particle Physics and Nuclear Physics, Faculty of Mathematics and Physics, Charles University,  V Hole\v{s}ovi\v{c}k\'{a}ch 2, 180 00 Prague 8, Czech Republic}
\email{ondra.hulik@gmail.com}
\author{Fridrich Valach}
\address{Department of Physics, Imperial College London\\Prince Consort Road, London, SW7 2AZ, UK}
\email{fridrich.valach@gmail.com}
\author{Daniel Waldram}
\address{Department of Physics, Imperial College London\\Prince Consort Road, London, SW7 2AZ, UK}
\email{d.waldram@imperial.ac.uk}
\usepackage[absolute]{textpos}
\begin{document}
\begin{textblock}{5}(14,2.5)
\noindent Imperial/TP/21/DW/1
\end{textblock}

\maketitle

\begin{abstract}
        We introduce the notion of $\ms G$-algebroid, generalising both Lie and Courant algebroids, as well as the algebroids used in $\ms E_{n(n)}\times\R^+$ exceptional generalised geometry for $n\in\{3,\dots,6\}$. Focusing on the exceptional case, we prove a classification of ``exact'' algebroids and translate the related classification of Leibniz parallelisable spaces into a tractable algebraic problem. After discussing the general notion of Poisson--Lie duality, we show that the Poisson--Lie U-duality is compatible with the equations of motion of supergravity.
\end{abstract}

\section{Introduction}
\subsection{String theory and Courant algebroids}
    When studying various aspects of string theory, Courant algebroids \cite{LWX} provide an invaluable tool. They can be seen as a ``many-points'' generalisation of quadratic Lie algebras (Lie algebras equipped with an inner product).\footnote{In a similar fashion, Lie algebroids can be seen as a ``many-points'' generalisation of Lie algebras.} More precisely, a Courant algebroid is given by a vector bundle $E\to M$ together with some extra structure, most notably a bracket on the space of sections. An important class is given by the so-called \emph{exact} Courant algebroids. These have $E\cong TM\oplus T^*M$ and are classified by $H^3(M)$, corresponding to the class of the 3-form flux in string theory \cite{Severa1,Severa2}. 
    
    More generally, one can encode all the bosonic NSNS field content of 10-dimensional type II supergravities by means of a \emph{generalised metric} on an exact Courant algebroid, and then describe the corresponding dynamics in terms of a suitable Ricci tensor and scalar curvature \cite{Siegel,HHZ,JLP,CSCW0}.\footnote{The RR fields can be seen as spinors w.r.t.\ the Courant description and thus they also fit into this framework, or alternatively as part of the generalised metric in exceptional generalised geometry. Their role in Poisson--Lie T-duality was elucidated in \cite{Falk}.} Extending to M-theory, the symmetries of eleven-dimensional supergravity on a $n$-dimensional manifold $M$, define an \emph{exceptional generalised geometry} \cite{Hull,PW,CSCW} in terms of a particular type of Leibniz algebroid \cite{Baraglia}. Again the bosonic fields define a generalised metric and the dynamics are encoded by the vanishing of a suitable Ricci tensor \cite{CSCW}.
    
\subsection{Poisson--Lie T-duality}
    Furthermore, the phenomenon of Poisson--Lie T-duality \cite{KS} turns out to be inherently linked with Courant algebroids \cite{Severa1}. This duality, which can be seen as a non-abelian generalisation of the usual stringy T-duality, relates two (or more) different string backgrounds, each described by an exact Courant algebroid with a generalised metric. In order for these algebroids to be Poisson--Lie dual to one another, they must be both \emph{pullbacks} of the same non-exact Courant algebroid. 
    
    Consequently, in order to fully understand the Poisson--Lie T-duality and its relation to supergravity, it was necessary to extend the relevant concepts (e.g.\ the curvature tensors) to the non-exact Courant setup \cite{GF,GF2,JV,SV0}. This result provided a simple proof \cite{SV0,SV} of the compatibility of the duality with string background equations, extending the result of \cite{VKS}. In addition, a method for searching for new solutions to these equations was developed.
    
    Finally, let us emphasise that the Poisson--Lie T-duality and its cousins studied in the present work are to be understood (in most Sections) in the most general sense of the word, i.e.\ including the \emph{spectator coordinates}. In the T-duality case this represents a vast generalisation of the most studied class of examples given by pairs of Poisson--Lie groups.
    
\subsection{Poisson--Lie U-duality}
    The $n$-torus compactifications of M-theory exhibit a U-duality symmetry, which features the split real forms of exceptional Lie groups of rank $n$ (as opposed to the split real form of the orthogonal group in the T-duality case). A Poisson--Lie-type generalisation of U-duality was first proposed and investigated in the case without spectators in \cite{Sakatani, MT}. One of the goals of this paper is to describe this phenomenon in the language of algebroids, allowing the employment of techniques and strategies known from Courant algebroids. In particular, this will involve defining a suitable non-exact generalisation of the algebroids that appear in exceptional generalised geometry.  
    
\subsection{Summary of results}
    In the present work we introduce a general framework of $\ms G$-algebroids, tailored for the study of dualities and related topics, such as Leibniz (or generalised) parallelisations. In addition to recovering the algebroids (up to $n=6$) used in exceptional generalised geometry, we recover Lie and Courant algebroids, and the algebroids in \cite{LSCW}. In each case we formulate the appropriate notion of Poisson--Lie duality.
    
    Focusing then on the exceptional case, we prove a classification result for exact algebroids (of ``M-theoretic type'') and reduce the classification of Leibniz parallelisable spaces to a quite simple algebraic problem. It should be noted that the latter essentially mirrors a result, derived using different methods, by Inverso \cite{Inverso} (see also \cite{BHL} for the $n=4$ case). We then provide a simple proof of the compatibility of the Poisson--Lie U-duality (in the general case with spectators) with the bosonic part of the equations of motion of the relevant supergravity (see Section \ref{sec:sugra}).
    
    It should also be emphasised that the presented framework is entirely geometric and avoids the need of an explicit coordinate description. Furthermore, it provides a natural language for the study of dualities at the level of algebroids, without the need for extending the spacetime.
    
\subsection{Outline of the paper}
    The paper is structured as follows. We start by discussing the general types of ``geometries'' (e.g.\ exceptional, Courant, etc.), encoded in an \emph{admissible group data set}, introduce (generalised) isotropic and coisotropic subspaces and provide several examples. We then define $\ms G$-algebroids and in particular exceptional algebroids, and discuss examples thereof together with some classification results in the exact case. Proceeding to pullbacks, we prove an important theorem concerning the construction of exceptional algebroids, and then turn to the related topic of Leibniz parallelisations. After discussing the general concept of Poisson--Lie duality, we again restrict our attention to the exceptional case, show how several simple examples from the literature fit into the present framework. We then prove the compatibility of the duality and supergravity equations of motion. Some technical details and proofs concerning the exceptional case are moved to the Appendix.
    
\subsection{Notation}
    We will denote Lie groups by $\ms G, \ms K, \ms{GL}(n,\R),\dots$, and their corresponding Lie algebras by $\mf g,\mf k,\mf{gl}(n,\R),\dots$ We will keep the same symbols $E,N,\dots$ for both group representation spaces and the corresponding associated bundles. The annihilator of a subspace $V\subset W$ will be denoted by $V^\circ\subset W^*$, the pairing of vectors with covectors by $\la\cdot,\cdot\ra$, and the transpose of a linear map by a superscript $t$. We shall write $S^2V$ for the second symmetric tensor power of a vector space $V$.
    
    In the text, we will be often working with maps $S^2E\to N$, $N\to S^2 E$, and their duals. For the sake of clarity we will not give these maps specific names, but will instead refer to them, and their (partial) duals, by a subscript -- for example, the image of a $\xi\otimes n$ under the map $E^*\otimes N\to E$ will be denoted by $(\xi\otimes n)_{E}$ (this map can be seen as the composition of $E^*\otimes N\to E^*\otimes S^2E \to E^*\otimes E\otimes E$ with the contraction of the first two terms).
    
\subsection{Outlook and future prospects}
    The present work opens the door for further investigations in the area of Poisson--Lie dualities or exceptional generalised geometry and its cousins. A natural question is the extension of the results to the case $n=7$ and beyond, as well as to geometries given by other groups, such as the $\ms{Spin}(n,n)\times\R^+$ of \cite{CSC}. One can also examine possible reformulations of the framework in terms of $L_\infty$-algebroids or dg manifolds, making a connection to the works \cite{Baraglia, Arvanitakis}. Furthermore, using the results of Section \ref{sec:class} one can try to perform a classification of Leibniz parallelisations (which in turn correspond to maximal consistent truncations \cite{LSCW}), or search for new Poisson--Lie U-dual backgrounds. A detailed study of these issues is left to later works.
    
\subsection{Acknowledgements}
    The authors would like to thank Florian Naef, Jan Slovák, Charles Strickland-Constable, and Pavol Ševera for helpful discussions and comments. The authors would also like to acknowledge the Higher Structures and Field Theory programme at the Erwin Schr\" odinger Institute in Vienna, Austria, in September 2020, where part of the research was carried out. Research of M.B. was supported by the GA\v{C}R Grant EXPRO 19-28628X. Research of O.H. was supported by the GA\v{C}R Grant EXPRO 20-25775X. F.V.\ was supported by the Early Postdoc Mobility grant P2GEP2\underline{\phantom{k}}188247 of the Swiss National Science Foundation. D.W.\ was supported in part by the STFC Consolidated Grants ST/P000762/1 and ST/T000791/1 and the EPSRC New Horizons Grant EP/V049089/1.
    
\section{Admissible group data set}
\subsection{Admissible group data set}
    There is a certain algebraic pattern underlying several types of ``geometries'', for example the geometry of Lie algebroids \cite{Pradines}, Courant algebroids \cite{LWX}, or the Leibniz algebroids appearing in the description of exceptional generalised geometry \cite{PW,Baraglia,CSCW}. Generalising this pattern, we introduce the concept of an \emph{admissible group data set}. For simplicity, we divide the definition in two parts.
    \begin{defn}
        A \emph{group data set} is given by a reductive (real) Lie group $\ms G$, a faithful representation $E$ of $\ms G$, a decomposition $S^2E=N\oplus\hat N$ into subrepresentations, and a map $N\to S^2 E$ proportional to the embedding.
    \end{defn} 
    Throughout the text we will use the following notation. By the map $N\to S^2E$ we will always mean the one from the definition, while $S^2E\to N$ will always be taken to be the projection w.r.t.\ $\hat N$.
    We also define a map $\pi'\colon \on{End}(E)\to \on{End}(E)$ to be the partial dual of the composition
    \[E\otimes E\to S^2 E\to N \to S^2 E\to E\otimes E,\] and we set $\pi:=1-\pi'$.
    \begin{defn}\label{def:grpdata}
        A group data set is \emph{admissible} if $\pi(\on{End}(E))\subset \mf g$.
    \end{defn}
    \begin{rem}
        This condition is based on a pattern common to various ``types'' of geometries (c.f.\ \cite{CSCW} and \cite{BCKT}). Note that  the components of the map $\pi'$ are usually denoted by $Y^{ij}_{\hphantom{ij}kl}$ in the exceptional field theory literature.
    \end{rem}
    \begin{lem}
        An admissible group data set is fully determined by specifying $\ms G$ and the decomposition $S^2E=N\oplus \hat N$ (i.e.\ the embedding $N\to S^2E$ is fixed), unless $\mf g\cong\mf{gl}(E)$ and $\hat N=0$. 
    \end{lem}
    \begin{proof}
        Given two different maps $N\to S^2E$ which both lead to $\pi(\on{End}(E))\subset \mf g$, we can take a suitable linear combination of the respective $\pi$'s to get that $\on{Im}(\on{id})=\on{End}(E)\subset \mf{g}$. Thus $\mf g\cong \mf{gl}(E)$, implying we have $N=0$ and $\hat N=S^2E$ (which makes $N\to S^2E$ trivial), or vice versa.
    \end{proof}
%From the lemma it follows that (unless $\mf g\cong\mf{gl}(E)$ and $\hat N=0$), the decomposition $S^2E=N\oplus \hat N$ uniquely determines $\lambda\in\R$ such that
    %$((n)_{E\otimes E})_N=\lambda n$.
%    Note that we can extend $E$ and $N$ to representations of the group $\ms G\times \R^+$, by letting $\R^+$ to act on $E$ and $N$ with weights $1$ and $2$, respectively. It is in fact the product $\ms G\times\R^+$, which will play the crucial role in our definition of a $\ms G$-algebroid.
    
\subsection{Examples}\label{subsec:ex}
    We now provide a list of examples, given by reductive groups with semisimple parts given by split real forms of simply-laced semisimple Lie groups. We provide also the characterisation of representations (of the semisimple part) in terms of Dynkin diagrams, with $E$ corresponding to a black node and $N$ to a blue one. %Finally, we specify the maps $S^2E\to N$ and $S^2E^*\to N^*$, as we will need them later.\footnote{The second map is the transpose of $N\to S^2E$.} 
    Notice that in these examples, the choices of $\ms G$, $E$ and $N\subset S^2E$ determine $\hat N$ as well. Consequently, we will often refer to an admissible group data set simply as a triple $(\ms G,E,N)$.
    
    \begin{example}\label{ex:alg_ordinary}
        The simplest example consists of $\ms G=\ms{GL}(n,\R)$, $N=0$ and $E$ the vector representation. Later it will give rise to Lie algebroids. The diagram is
        \[\begin{tikzpicture}[scale=0.6, baseline=-0.6ex, thick]
            \draw (0,0) -- (1.6,0); \draw (2.4,0) -- (3,0); \filldraw [black] (1.8,-.2) circle (.5pt); \filldraw [black] (2,-.2) circle (.5pt); \filldraw [black] (2.2,-.2) circle (.5pt);
            \node at (0,0) [proj1] {\hphantom{+}}; \node at (1,0) [proj0] {\hphantom{+}}; \node at (3,0) [proj0] {\hphantom{+}};
        \end{tikzpicture}\qedhere\]
    \end{example}
    \begin{example}\label{ex:alg_onn}
        Take $\ms G=\ms{O}(n,n)$, with the vector representation $E= \R^{2n}$, and $N\cong\R$. This gives
        
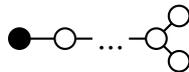
\[\begin{tikzpicture}[scale=0.6, baseline=-0.6ex, thick]
            \draw (0,0) -- (1.6,0); \draw (2.4,0) -- (3,0); \draw (3,0) -- (3.5,.5); \draw (3,0) -- (3.5,-.5);
            \filldraw [black] (1.8,-.2) circle (.5pt); \filldraw [black] (2,-.2) circle (.5pt); \filldraw [black] (2.2,-.2) circle (.5pt);
            \node at (0,0) [proj1] {\hphantom{+}}; \node at (1,0) [proj0] {\hphantom{+}}; \node at (3,0) [proj0] {\hphantom{+}}; \node at (3.5,.5) [proj0] {\hphantom{+}}; \node at (3.5,-.5) [proj0] {\hphantom{+}}; 
        \end{tikzpicture}\qedhere\]
        The induced map $S^2E\to \R$ corresponds to an inner product on $E$ of signature $(n,n)$. Clearly, the setup can be generalised to the $\ms{O}(p,q)$-case. This will correspond to Courant algebroids.
    \end{example}
    \begin{example}\label{ex:exc_data}
        Take $n\in\{3,\dots,6\}$ and let $\ms G=\ms E_{n(n)}\times\R^+$, corresponding to the split real form of the exceptional Lie algebra $\mf e_n$.\footnote{Strictly speaking, we only get an exceptional Lie algebra for $n=6$. The remaining cases are defined by following the pattern of Dynkin diagrams.} The details about the representations $E$, $N$ as well as the structure of the algebras can be found in the Appendix. Here it suffices to say that under the subalgebra $\mf{gl}(n,\R)\subset \mf e_{n(n)}\oplus \R$, defining $T:=\R^n$, we have 
        \[E\cong T\oplus \w{2}T^*\oplus \w{5}T^*,\quad N\cong T^*\oplus \w{4}T^*\oplus(T^*\otimes\w{6}T^*),\]
        while $\R^+$ acts on $E$ and $N$ with weights 1 and 2, respectively.
        Writing $u=X+\sigma_2+\sigma_5\in E$, the map $S^2E\to N$ is given by
        \[u\otimes u\mapsto 2\,i_X \sigma_2+(2\,i_X \sigma_5-\sigma_2\wedge\sigma_2)+2\,j\sigma_2\wedge\sigma_5,\]
        where $(j\sigma_2\wedge \sigma_5)(Y):=(i_Y\sigma_2)\wedge\sigma_5$ for $Y\in T$. The map $S^2E^*\to N^*$, which is dual to $N\to S^2E$, is given (up to a multiple) by an analogous formula. We shall refer to this data as the \emph{exceptional (admissible) group data set}. In terms of the Dynkin diagrams, we get
        \[\begin{tikzpicture}[scale=0.6, baseline=-0.6ex, thick]
        \draw (5,-1) -- (5,2);
        \draw (0,0) -- (4,0); \draw (2,1) -- (2,0); \node at (2,1) [proj0] {\hphantom{+}}; \node at (0,0) [proj1] {\hphantom{+}}; \node at (4,0) [proj2] {\hphantom{+}};
        \foreach \x in {1,...,3}
            \node at (\x,0) [proj0] {\hphantom{+}};
        \end{tikzpicture}\quad
        \begin{tikzpicture}[scale=0.6, baseline=-0.6ex, thick]
        \draw (4,-1) -- (4,2);
        \draw (0,0) -- (3,0); \draw (1,1) -- (1,0); \node at (1,1) [proj0] {\hphantom{+}}; \node at (0,0) [proj1] {\hphantom{+}}; \node at (3,0) [proj2] {\hphantom{+}};
        \foreach \x in {1,...,2}
            \node at (\x,0) [proj0] {\hphantom{+}};
        \end{tikzpicture}\quad
        \begin{tikzpicture}[scale=0.6, baseline=-0.6ex, thick]
        \draw (3,-1) -- (3,2);
        \draw (0,0) -- (2,0); \draw (0,1) -- (0,0); \node at (0,1) [proj0] {\hphantom{+}}; \node at (0,0) [proj1] {\hphantom{+}}; \node at (2,0) [proj2] {\hphantom{+}};
        \foreach \x in {1,...,1}
            \node at (\x,0) [proj0] {\hphantom{+}};
        \end{tikzpicture}\quad
        \begin{tikzpicture}[scale=0.6, baseline=-0.6ex, thick]
        \draw (1,0) -- (2,0); \node at (0,1) [proj1] {\hphantom{+}}; \node at (1,0) [proj1] {\hphantom{+}}; \node at (2,0) [proj2] {\hphantom{+}};
        \end{tikzpicture}\]
    \end{example}
    \begin{example}[\cite{LSCW}]
        Let $n\ge 2$. Consider $\ms G=\ms{SL}(n+1,\R)\times\R^+$, with $E=\w{2}\R^{n+1}$, $N=\w{4}\R^{n+1}$ and $\R^+$ acting with weights 1 and 2. If $n=2$, $n=3$, and $n=4$, we recover special cases of the first, second, and third examples, respectively. For $n>3$ this corresponds to
        \[\begin{tikzpicture}[scale=0.6, baseline=-0.6ex, thick]
            \draw (0,0) -- (3.6,0); \draw (4.4,0) -- (5,0); \filldraw [black] (3.8,-.2) circle (.5pt); \filldraw [black] (4,-.2) circle (.5pt); \filldraw [black] (4.2,-.2) circle (.5pt);
            \node at (0,0) [proj0] {\hphantom{+}}; \node at (1,0) [proj1] {\hphantom{+}}; \node at (2,0) [proj0] {\hphantom{+}}; \node at (3,0) [proj2] {\hphantom{+}}; \node at (5,0) [proj0] {\hphantom{+}};
        \end{tikzpicture}\qedhere\]
        The maps $S^2E\to N$ and $S^2E^*\to N^*$ are proportional to the wedge product. 
    \end{example}
%    \begin{rem}
%        Note that in all the above examples, the subalgebra of $\mf{gl}(E)$ which preserves the splitting $S^2E=N\oplus \bar N$ is given precisely by $\mf g\oplus \R$. This is a consequence of the fact that 
%        \[\on{End}(E)\cong E^*\otimes E\cong \R\oplus \mf g\oplus R,\] 
%        with $R$ irreducible.
%    \end{rem}
%    \begin{example}[\cite{CSC}]
%        For the case where $N$ is reducible, consider $\ms G=\ms{Spin}(n,n)\times\R^+$. We take $E$ to be the spinor representation and $N$ to be the sum of the fundamental representations corresponding to the cyan nodes in the diagram, plus an extra $\R$ if $n$ is a multiple of $4$,\footnote{Every fourth node in the horizontal string on the diagram is cyan.} while $\R$ acts again with weights 1 and 2. If $n=3$, $n=4$, and $n=5$, we recover special cases of the first, second, and third examples, respectively.
%        \[\begin{tikzpicture}[scale=0.6, baseline=-0.6ex, thick]
%        \draw (0,0) -- (8.6,0); \draw (1,1) -- (1,0); \node at (1,1) [proj0] {\hphantom{+}}; \node at (0,0) [proj1] {\hphantom{+}};
%        \foreach \x in {1,2,4,5,6,8}
%            \node at (\x,0) [proj0] {\hphantom{+}};
%        \node at (3,0) [proj3] {\hphantom{+}}; \node at (7,0) [proj3] {\hphantom{+}};
%        \filldraw [black] (8.8,-.2) circle (.5pt); \filldraw [black] (9,-.2) circle (.5pt); \filldraw [black] (9.2,-.2) circle (.5pt);
%        \end{tikzpicture}\]
%    \end{example}
    \begin{rem}
        The $\ms O(n,n)$ example above differs from the rest by not having an extra central factor. Even though it is this semisimple choice that gives rise to Courant algebroids, one can also consider the analogous $\ms O(n,n)\times\R^+$-geometry, as in \cite{CSCW0}. Physically, this has the advantage of treating the entire NSNS sector, including the dilaton, in a uniform way.
    \end{rem}
\subsection{Isotropy and coisotropy}
    We now proceed to the introduction of isotropic and coisotropic subspaces, which (especially the latter one) will be important in the subsequent sections. This will generalise the usual notions from Riemannian geometry.
    \begin{defn}
        We say that a subspace $V\subset E$ is \emph{isotropic} if $(V\otimes V)_N=0$.
        Similarly, we say that a subspace $V\subset E$ is \emph{coisotropic} if $(V^\circ\otimes V^\circ)_{N^*}=0$. A subspace $V\subset E$ is \emph{Lagrangian} if it is maximally isotropic (cannot be further enlarged). Similarly, a subspace $V\subset E$ is \emph{co-Lagrangian} if it is minimally coisotropic (has no proper coisotropic subalgebra).
    \end{defn}
    \begin{rem}
       In the language of double/exceptional field theory, the coisotropic subspaces correspond to solutions of the section constraint \cite{Siegel,HHZ,CSCW}. Note that not all (co-)Lagrangian subspaces of a given $E$ need to have the same dimension. For instance, as shown in Proposition \ref{prop:classcoiso}, in the exceptional case there are 2 possible co-Lagrangian subspaces (up to an isomorphism), corresponding to the M-theory and type IIB solutions of the section constraint.\footnote{The type IIA solutions correspond instead to certain non-minimally coisotropic subspaces.}
    \end{rem}
    \begin{example}
        In the $(\ms{GL}(n,\R),\R^n,0)$ case, any subspace is coisotropic, while the only co-Lagrangian subspace is $0$.
    \end{example}
    \begin{example}
        In the $(\ms O(n,n),\R^{2n},\R)$ case, the space $E=\R^{2n}$ is equipped with an inner product of signature $(n,n)$. The coisotropy (and isotropy) coincide with the usual notions, w.r.t.\ this structure; Lagrangian and co-Lagrangian subspaces are the same, and they are both half-dimensional.
    \end{example}
    \begin{example}\label{ex:exc_meta}
        In the $(\ms{SL}(n+1,\R)\times\R^+,\w{2}\R^{n+1},\w{4}\R^{n+1})$ case for $n>3$, there are precisely two types of co-Lagrangian subspaces:
        \begin{enumerate}
            \item $V=\w{2}U\subset E$, with $U\subset \R^{n+1}$ a subspace of codimension $1$ ($V$ has codimension $n$),
            \item $V=(\w{2}\Xi)^\circ\subset E$, with $\Xi\subset (\R^{n+1})^*$ of dimension $3$ ($V$ has codimension $3$).
        \end{enumerate}
    \end{example}
    \begin{lem}
        $V$ is coisotropic iff $(V^\circ\otimes N)_E\subset V$.
    \end{lem}
    \begin{proof}
        $(V^\circ \otimes V^\circ)_{N^*}=0\iff \la V^\circ \otimes V^\circ,N\ra=0\iff \la (V^\circ\otimes N)_E,V^\circ\ra=0$.\qedhere
    \end{proof}
    \begin{lem}\label{lem:colagexc}
         For the exceptional group data set, $(V^\circ\otimes N)_E=V$ iff $V$ is co-Lagrangian. 
    \end{lem}
    \begin{proof}
        Supposing $(V^\circ\otimes N)_E=V$, $V$ is clearly coisotropic. If there is a proper coisotropic subspace $V'\subsetneq V$, then $(V^\circ\otimes N)_E\subset (V'^\circ\otimes N)_E\subset V' \subsetneq V$. The other direction follows from an explicit check, using the classification of (co-)Langrangian subspaces in Appendix \ref{ap:class}.
    \end{proof}
%    Note that the first implication is true for any admissible group data set.
\section{\texorpdfstring{$\ms G$}{G}-algebroids}

Let us now define the algebroid structure that we will use to unifies the study the exceptional and other geometries.\footnote{Strictly speaking, we shall also assume a choice of a $\ms G$-structure (i.e.\ a set of $\ms G$-related frames) on the representation $E$. (Note that because of the construction, all of the above examples have a natural such structure.) This will induce a $\ms G$-structure on the corresponding associated bundle.}

\begin{defn}\label{def:galg}
    Fix an admissible group data set. A \emph{$\ms G$-algebroid} consists of a principal $\ms G$-bundle over $M$ together with the following structure on the associated vector bundles $E\to M$ and $N\to M$:
    \begin{itemize}
        \item an $\R$-linear bracket $[\cdot,\cdot]\colon \Gamma(E)\times \Gamma(E)\to \Gamma(E)$
        \item a vector bundle map $\rho\colon E\to TM$ (the \emph{anchor})
        \item an $\R$-linear operator $\mc D\colon \Gamma(N)\to\Gamma(E)$
%        \item a class of $\ms G$-related frames on $E$ (a $\ms G$-structure)
    \end{itemize}
    such that for any $u,v,w\in\Gamma(E)$, $n\in\Gamma(N)$, $f\in C^\infty(M)$,
    \begin{align}
        [u,[v,w]]&=[[u,v],w]+[v,[u,w]]\vphantom{\dh} \label{eq:jacobi}\\
        [u,fv]&=f[u,v]+(\rho(u)f)v \vphantom{\dh} \label{eq:anchor}\\
        [u,v]+[v,u]&=\mc D (u\otimes v)_N \vphantom{\dh} \label{eq:sym}\\
        \mc D(fn)&=f\mc D n+(\dh f\otimes n)_E, \label{eq:diff}
    \end{align}
    where $\dh:=\rho^t \circ d\colon C^\infty(M)\to \Gamma(E^*)$, and the action $[u,\cdot]$ preserves the $\ms G$-structure.
    
    A $\ms G$-algebroid with $M=\text{pt}$ is called a \emph{$\ms G$-algebra}.
\end{defn}
%    We shall refer to $\ms G$-algebroids as $E\to M$ or simply as $E$.
    The last condition (the bracket preserving the $\ms G$-structure), is to be understood as follows. Condition \eqref{eq:anchor} implies that any $u$ gives rise to a vector field on the total space $E$, which projects onto $\rho(u)$. Lifting this vector field to the frame bundle of $E$, the condition asks that it preserves the given $\ms G$-subbundle. In other words, for any trivialisation by a $\ms G$-frame we have that the vertical part of $[u,\cdot]$ acts as an element in the adjoint representation.
    
    In particular this means that the action $[u,\cdot]$ can be extended to other bundles associated to representations of $\ms G$, in particular to tensor powers of $E$ and their duals. For instance, for $v,w\in\Gamma(E)$ and $\xi\in\Gamma(E^*)$, we have 
    \[[u,v\otimes w]=[u,v]\otimes w+v\otimes[u,w],\qquad \rho(u)\la\xi,v\ra=\la[u,\xi],v\ra+\la \xi,[u,v]\ra.\]
    
    Notice also that since the map $S^2 E\to N$ is surjective, the operator $\mc D$ is uniquely determined by condition \eqref{eq:sym}. %However, it is still convenient to keep it as a part of the data in the definition.
\begin{lem}\label{lem:basic}
     For a $\ms G$-algebroid we have for all $u,v\in\Gamma(E)$, $n\in\Gamma(N)$, $f\in C^\infty(M)$
     \begin{align}
        \rho([u,v])&=[\rho(u),\rho(v)],\vphantom{\dh}\tag{a}\label{eq:morphism}\\
        [\mc D n,u]&=0,\tag{b}\vphantom{\dh}\label{eq:lcentre}\\
        \rho\circ \mc D&=0,\vphantom{\dh}\quad
        \on{Ker}(\rho)\text{ is coisotropic},\vphantom{\dh}\tag{c}\label{eq:coiso}\\
        [fu,v]&=f[u,v]-\pi(\dh f\otimes u) v,\tag{d}\label{eq:first_slot}\\
        [u,\mc Dn]&=\mc D[u,n],\tag{e}\vphantom{\dh}\label{eq:Dvsbkt}\\
        [u,\dh f]&=\hat d(\rho(u)f),\tag{f}\label{eq:fctionvsbkt}
     \end{align}
\end{lem}
\begin{proof}
    First equation is obtained by setting $w\to f w$ in \eqref{eq:jacobi} and repeatedly using \eqref{eq:anchor}.
    Setting $u=v$ in \eqref{eq:jacobi}, using \eqref{eq:sym} and the surjectivity of $S^2 E\to N$, we get \eqref{eq:lcentre}. 
    The third line is obtained by acting with the anchor on \eqref{eq:sym} and \eqref{eq:diff} and using the fact that $(\on{Ker}\rho)^\circ\cong \on{Im}(\rho^t)$.
    The next claim follows by a straightforward application of \eqref{eq:anchor}, \eqref{eq:sym}, \eqref{eq:diff}, and
    \[-(\rho(v)f)u+[\dh f\otimes ( u\otimes v )_N]_E=-\la\dh f,v\ra u+[\dh f\otimes ( u\otimes v )_N]_E=-\pi(\dh f\otimes u) v.\] 
    For \eqref{eq:Dvsbkt} and \eqref{eq:fctionvsbkt} we calculate
    \begin{align*}
        [u,\mc D(v\otimes w)_N]&=[u,[v,w]+[w,v]]=[[u,v],w]+[v,[u,w]]+[[u,w],v]+[w,[u,v]]\\
        &=\mc D([u,v]\otimes w)_N+\mc D(v\otimes[u,w])_N=\mc D[u,(v\otimes w)_N],
    \end{align*}
    \begin{equation*}
        \la [u,\dh f],v\ra=\rho(u)\la \dh f,v\ra-\la \dh f,[u,v]\ra=\rho(u)\rho(v)f-\rho([u,v])f=\rho(v)\rho(u)f=\la \hat d(\rho(u)f),v\ra.\qedhere
    \end{equation*}
\end{proof}
Since $\on{Ker}(\rho)$ is coisotropic, we have a chain complex
\begin{equation}\label{eq:seq}
    T^*M\otimes N \to E\xrightarrow{\rho} TM\to 0.
\end{equation}

    \begin{defn}
        We say that a $\ms G$-algebroid is \emph{exact} if this is an exact sequence (i.e.\ it is exact at $E$ and $TM$). More generally, a $\ms G$-algebroid with $\rho$ surjective is called \emph{transitive}.
    \end{defn}
    
\section{Examples of \texorpdfstring{$\ms G$}{G}-algebroids}
\subsection{Lie algebroids}
    Taking $(\ms G,E,N)=(\ms{GL}(n,\R),\R^n,0)$ we get the definition of Lie algebroids \cite{Pradines}. In this case a $\ms G$-algebra is the same as a Lie algebra.
    
    One of the simplest examples is:
    \begin{example}[\emph{Tangent Lie algebroid}]
        Take $E=TM$, with $M$ an arbitrary manifold, the bracket given by the commutator of vector fields, and the anchor being the identity (we have $\mc D=0$).
    \end{example}
    The sequence \eqref{eq:seq} becomes simply $0\to E\to TM\to 0$, implying:
    \begin{prop}
        In the $(\ms{GL}(n,\R),\R^n,0)$-case, a $\ms G$-algebroid is exact iff it is a tangent Lie algebroid.
    \end{prop}
    
\subsection{Courant algebroids}
    Courant algebroids \cite{LWX} correspond to $\ms G$-algebroids for $(\ms O(p,q), \R^{p+q}, \R)$ and with $\mc D=\hat d$. In the last equality, we use the identification $E\cong E^*$ provided by the $\ms{O}(p,q)$-structure.
    
    \begin{example}[\emph{Twisted generalised tangent bundle}]\label{ex:courant}
        Let $M$ be an $n$-dimensional manifold and $H\in \Omega^3(M)$ a closed form. Then \[E=TM\oplus T^*M\] has a natural $\ms O(n,n)$-structure, given by the pairing of vectors and 1-forms. The bracket is given by
        \[[X+\alpha,X'+\alpha']=\mc L_X X'+(\mc L_X \alpha'-i_{X'}d\alpha+i_Xi_{X'}H),\]
        the anchor is the projection onto $TM$, and $\mc D=d$. 
    \end{example}
    The classification result \cite{Severa1} for exact Courant algebroids can be stated in the present language as follows.
    \begin{thm}
        In the $(\ms O(p,q), \R^{p+q}, \R)$-case, a $\ms G$-algebroid is exact iff it has the form of a twisted generalised tangent bundle. Exact $\ms G$-algebroids only exist for $p=q$ and they are classified by $H^3(M)$. Locally, every exact $\ms G$-algebroid has the form from Example \ref{ex:courant} with $H=0$.
    \end{thm}

\subsection{Elgebroids}
    \begin{defn}
        An \emph{exceptional algebroid}, or simply \emph{elgebroid}, is a $\ms G$-algebroid given by the exceptional group data set (see Example \ref{ex:exc_data}), for some $n\in\{3,\dots,6\}$. An elgebroid over a point is called an \emph{elgebra}.
    \end{defn}
    \begin{defn}
        An elgebroid is \emph{M-exact} if it is exact and $\dim M=n$. It is \emph{IIB-exact} if it is exact and $\dim M=n-1$.
    \end{defn}
    \begin{rem}
       As the name suggests, these are related to M-theory and type IIB solutions of the section constraint. Proposition \ref{prop:classcoiso} shows that these are the only 2 possibilities of exact elgebroids.
    \end{rem}
    \begin{example}[\emph{Exceptional tangent bundle} \cite{Hull,PW,CSCW}]\label{ex:exc}
        Let $M$ be a manifold of dimension $n\in\{3,\dots,6\}$. We can then consider
        \begin{equation}\label{eq:etb}
            E:=TM\oplus \w{2}T^*M\oplus \w{5}T^*M,
        \end{equation}
        with $\rho$ given by the projection onto the first factor and
        \begin{equation}\label{eq:bracket}
            [X+\sigma_2+\sigma_5,X'+\sigma_2'+\sigma_5']=\mc L_X X'+(\mc L_X \sigma_2'-i_{X'}d\sigma_2)+(\mc L_X \sigma_5'-i_{X'}d\sigma_5-\sigma_2'\wedge d\sigma_2).
        \end{equation}
        The map $\mc D$, acting on the sections of $N=T^*M\oplus \w{4}T^*M\oplus(T^*M\otimes\w{6}T^*M)$, coincides with $d$ on the first two summands, and vanishes on the third.\footnote{The $\ms E_{n(n)}\times\R^+$-structure is induced from the $\ms{GL}(n,\R)$-structure of the bundle $E$, given by the decomposition \eqref{eq:etb} (c.f. Example \ref{ex:exc_data}).} This is an M-exact elgebroid.
    \end{example}
    Conversely, we have:
    \begin{thm}\label{thm:class}
        Any M-exact elgebroid is locally isomorphic to the exceptional tangent bundle.
    \end{thm}
    The proof of the Theorem is in Appendix \ref{ap:proof_class}.
    \begin{rem}
       It is clear from the proof that the bracket can be twisted, in analogy to Example \ref{ex:courant}, using a pair $F_1\in\Omega^1(M)$, $F_4\in\Omega^4(M)$ satisfying $dF_1=0$ and $dF_4+F_1\wedge F_4=0$. However, a global classification of exact elgebroids is more subtle than in the Courant case. Note that in the physics literature, $F_1$ is taken to be exact, since otherwise, given a choice of generalised metric, the warp factor of the $(11-n)$-dimensional part of the M-theory metric will not be globally defined. 
    \end{rem}

\section{Pullbacks}\label{sec:pullbacks}
    We now proceed to define pullbacks, which play an important role in the construction of $\ms G$-algebroids and in the description of dualities. This can be seen as an extension of the results obtained for the Courant case in \cite{LBM}.
\begin{defn}
    Fix an admissible group data set. Let $\varphi\colon M' \to M$ be a surjective submersion and $E\to M$ a $\ms G$-algebroid. A $\ms G$-algebroid structure on $E':=\varphi^*E\to M'$, with the induced $\ms G$-structure, is called a \emph{pullback of $E$} if for all sections $u,v\in\Gamma(E)$ and $n\in\Gamma(N)$ we have
    \[[\varphi^*u,\varphi^*v]'=\varphi^*[u,v],\quad \varphi_*\rho'(\varphi^*u)=\rho(u),\quad \mc D'\varphi^*n=\varphi^*\mc Dn.\]
\end{defn}
    Note that the $\ms G$-algebroid structure on $E'$ is fully determined by its anchor, the map $\varphi$, and the structures on $E$. Thus, specifying the anchor, there is at most one pullback (for a given $E$ and $\varphi$).
\begin{defn}
    A $\ms G$-algebroid is \emph{Leibniz parallelisable} if it can be written as a pullback of a $\ms G$-algebra (along the unique map $\varphi$ to the point).
\end{defn}
    This coincides with the notion of Leibniz (or general) parallelisability \cite{LSCW} in the physics literature -- being a pullback of a $\ms G$-algebra means that there is a global $\ms G$-frame $e_\alpha$ of $E$ for which the structure coefficients $c^\gamma_{\alpha\beta}$, defined by $[e_\alpha,e_\beta]=c^\gamma_{\alpha\beta}e_\gamma$, are constant.
\begin{defn}
    An \emph{action} of a $\ms G$-algebra $E$ on a manifold $M'$ is a map $\chi\colon E\to \Gamma(TM')$ which preserves the brackets. The \emph{stabiliser} of the action at a point $p\in M'$ is the kernel of $\chi_p\colon E\to T_pM'$.
\end{defn}
    In the particular case $M=\text{pt}$, the anchor of $E'$ can be seen as an action of $E$ on $M'$. A natural question then is: Given an action of a $\ms G$-algebra on a manifold, when does this define a $\ms G$-algebroid via the pullback construction? Let us now answer the question for the case of M-exact elgebroids.
\begin{thm}\label{thm:action}
    A transitive action of an elgebra $E$ on an $n$-dimensional manifold $M'$ defines an M-exact pullback elgebroid on $E'=M'\times E$ iff at every point the stabilisers are co-Lagrangian of codimension $n$.
\end{thm}
    The proof can be found in Appendix \ref{ap:proof_action}. Note that an analogous result was obtained (using different methods) in \cite{Inverso}.
    
    More generally, a necessary condition for an action of a $\ms G$-algebra to define a $\ms G$-algebroid is that the stabilisers are coisotropic. Setups where the coisotropy condition is also sufficient include Lie and Courant algebroids \cite{LBM}.
%\begin{example}
%    In the Lie algebroid case $(\ms{GL}(n,\R),\R^n,0)$, the conditions of the Theorem are trivially satisfied. Thus any action of a Lie algebra $\mf k$ on $M'$ induces an \emph{action Lie algebroid} $\mf k\times M'\to M'$.
%\end{example}
%\begin{example}
%    In the $(\ms O(p,q), \R^{p+q}, \R)$ setup, consider a $\ms G$-algebra given by a quadratic Lie algebra with $\mc D=0$.\footnote{i.e. a Courant algebroid over a point} We see that an action of this Lie algebra on a manifold gives rise to a $\ms G$-algebroid iff the stabilisers are coisotropic at every point. The resulting algebroid is again Courant (has $\mc D=\hat d$), and is called an \emph{action Courant algebroid}. This recovers a result from \cite{LBM}.
%\end{example}

\section{Classification of exact Leibniz parallelisable elgebroids}\label{sec:class}
    Let $E$ be an elgebra. Since \[[\on{Im}(\mc D),E]=0,\quad [E,\on{Im}(\mc D)]\subset \on{Im}(\mc D),\] 
    we have that $\on{Im}(\mc D)\subset E$ is a two-sided ideal, and so we can construct a Lie algebra \[\mf g_E:=p(E),\] where $p$ is the projection $E\to E/\on{Im}(\mc D)$.
    More generally, if $V\subset E$ is a subalgebra then $\mf g_V:=p(V)\subset \mf g_E$ is a Lie subalgebra. Conversely, if $\mf h\subset \mf g_E$ is a Lie subalgebra, then $p^{-1}(\mf h)\subset E$ is a subalgebra. 
    
    We shall denote the 1-connected Lie group corresponding to $\mf g_E$ by $\ms G_E$, and we will denote by $\ms G_V\subset \ms G_E$ a subgroup corresponding to $\mf g_V$ for the subalgebra $V\subset E$.
    Note that both $E$ and $N$ are $\ms G_E$-modules and the action of $\ms G_E$ preserves the bracket, the map $\mc D$, as well as the maps between $S^2E$ and $N$.
\begin{thm}
    Let $E$ be an elgebra and $V\subset E$ a co-Lagrangian subalgebra of codimension $n$, satisfying $\on{Im}\mc D\subset V$.
    Suppose $\ms G_V\subset \ms G_E$ is closed. The natural action of $E$ on $M':=\ms G_E/\ms G_V$ then gives rise to an M-exact Leibniz parallelisable elgebroid over $M'$. Every M-exact Leibniz parallelisable elgebroid over a connected compact base arises in this way, for some pair $(E,V)$.\footnote{Strictly speaking, we also need to make a choice of the subgroup $\ms G_V\subset \ms G_E$, corresponding to a fixed Lie algebra $\mf g_V$ (this is a discrete choice). For example, when $\mf g_V=0$, we can take $\ms G_V$ to be a discrete subgroup of $\ms G_E$.}
\end{thm}
\begin{proof}
    The Lie algebra $\mf g_E$ acts on $\ms G_E/\ms G_V$, with the stabiliser at point $[g^{-1}]$ given by $\on{Ad}_g\mf g_V$. This lifts to an action $\chi$ of $E$ on $M'$, with the stabiliser $g\cdot V$, where $g\cdot$ denotes the action of $g\in\ms G_E$. We see that $g\cdot V$ is co-Lagrangian iff $V$ is.
    
    Conversely, if $E'\to M'$ is M-exact and Leibniz parallelisable, coming from some elgebra $E$, then the anchor gives a transitive action of $\mf g_E$ on $M'$ (because $\on{Im}\mc D$ acts trivially). Since $M'$ is compact, $M'=\ms G_E/\ms H$ for some $\mf h\subset \mf g_E$, yielding $V=p^{-1}(\mf h)$.
\end{proof}

Using this result, a classification of Leibniz parallelisable M-exact elgebroids translates into a tractable algebraic problem and thus becomes an achievable goal. (It might still require some case-to-case analysis.) We leave this problem to a later work.

More generally, one easily sees that any transitive/exact Leibniz parallelisable $\ms G$-algebroid over a connected compact base arises from some pair of a $\ms G$-algebra and a coisotropic/co-Lagrangian subalgebra thereof. (However, in general not every such pair gives rise to a $\ms G$-algebroid.)
%If $\on{codim}V=n$, it corresponds to the classification of $\ms G$-parallelisable (twisted) exceptional tangent bundles; for $\on{codim}V=n-1$ it relates to the corresponding classification of ``type IIB bundles'' \cite{Hull}. 
%\begin{rem}
%   The type IIA case is related to non-minimal coisotropic $V$'s of codimension $n-1$, resulting in a special class of non-exact transitive elgebroids.
%\end{rem}

\section{Poisson--Lie duality}
We now use pullbacks to define a general notion of Poisson--Lie duality, extending the definition from \cite{SV0}.
\begin{defn}
    A pair of exact $\ms G$-algebroids, which are both pullbacks of a given $\ms G$-algebroid $E\to M$, are said to be (mutually) \emph{Poisson--Lie dual}.
    If $M\neq\on{pt}$ and $M=\on{pt}$, this is a Poisson--Lie duality \emph{with} and \emph{without spectators}, respectively. (The manifold $M$ is called the \emph{manifold of spectators}.)
\end{defn}
\begin{example}
    In the Courant algebroid case we recover the \emph{Poisson--Lie T-duality} of \cite{KS}, while the exceptional case gives the \emph{Poisson--Lie U-duality}, introduced in the case without spectators in \cite{Sakatani,MT}.
\end{example}
Let us now discuss some examples of Poisson--Lie duality without spectators. This corresponds to pairs of Leibniz parallelisable exact $\ms G$-algebroids arising from the same $\ms G$-algebra $E$, but different co-Lagrangian subalgebras $V$.
\begin{example}
    In the Lie algebroid case, the Poisson--Lie duality without spectators is trivial in the sense that any given $\ms G$-algebra admits a unique (trivial) co-Lagrangian $V$.
\end{example}
\begin{example}
    In the Courant algebroid case, the Poisson--Lie duality without spectators corresponds to different choices of Lagrangian subalgebras $\mf h,\mf h'\subset \mf g$. In the special case when $\mf h\cap \mf h'=0$, the corresponding groups $\ms H$ and $\ms H'$ carry compatible Poisson structures, i.e.\ they are Poisson--Lie groups. This is the origin of the term ``Poisson--Lie (T-)duality''.
\end{example}

\section{Examples of Leibniz parallelisable M-exact elgebroids}
        We now provide a short list of examples of Leibniz parallelisable M-exact elgebroids (resp.\ exceptional tangent bundles), arising from a pair of an elgebra $E$ and its co-Lagrangian subalgebra $V$ of codimension $n$, satisfying $\on{Im}\mc D\subset V$.
    \subsection{U-duality}    
        The simplest case is the one with $E$ an abelian Lie algebra, with $\mc D=0$. Taking $\ms G_E$ to be the $(\dim E)$-dimensional torus and $V$ to be a co-Lagrangian subspace of codimension $n$ corresponding to a closed sub-torus $\ms G_V\subset \ms G_E$, gives rise to a Leibniz parallelisable exceptional tangent bundle on the $n$-dimensional torus $T^n\cong\ms G_E/\ms G_V$. Different $V$'s are related by an $\ms E_{n(n)}\times\R^+$-transformation and give rise to Poisson--Lie dual setups. This is the standard \emph{U-duality} from string theory. Note that although the dual spaces will be isomorphic as manifolds, equipping $E$ with a generalised metric (see the next Section) will result in different sets of geometric data.
    \subsection{Group manifolds}\label{sec:grp}
        More generally, a class of examples of Leibniz parallelisable exceptional tangent bundles (known as ``Scherk--Schwarz'' reductions in the physics literature) arises from group manifolds, using the trivialisation of the tangent bundle by left-invariant vector fields~\cite{CSCW}. The corresponding pair $(E,V)$ is given as follows.
            
        Consider a Lie algebra $\mf k$, corresponding to a Lie group $K$, with $\dim\mf k\in\{3,\dots,6\}$, and take% and choose an element $F_4\in\w{4}\mf k^*$ closed under the Chevalley-Eilenberg differential. Take
        \[E:=\mf k\oplus \w{2}\mf k^*\oplus\w{5}\mf k^*, \quad V:=\w{2}\mf k^*\oplus\w{5}\mf k^*,\]
       %with the Leibniz bracket given by 
        \[[X+\sigma_2+\sigma_5,X'+\sigma_2'+\sigma_5']=\on{ad}_X X'+(\on{ad}_X \sigma_2'-i_{X'}\delta\sigma_2)+(\on{ad}_X \sigma_5'-i_{X'}\delta\sigma_5-\sigma_2'\wedge \delta\sigma_2),\]
        where $\delta$ is the Chevalley-Eilenberg differential and $\on{ad}$ denotes the action of $\mf k$ (on $\mf k$ and on $\w{\bullet} \mf k^*$). %This gives rise to an exceptional tangent bundle over the group manifold $K$.
            
        This can be modified by taking an arbitrary pair of elements $F_1\in \mf k^*$, $F_4\in\w{4}\mf k^*$, satisfying $\delta F_1=0$ and $\delta F_4+F_1\wedge F_4=0$, and using the analogue of formula \eqref{eq:long}. As a result we again obtain an M-exact elgebroid on $K$.
        
        More generally, a rich class of Leibniz parallelisations on groups equipped with non-invariant structures has been constructed and studied in \cite{Sakatani, MT}. This provided one of the motivations for the present work.
        
    \subsection{4-sphere}%The case \texorpdfstring{$n=4$}{n=4}}
        Let us now describe how the example of $S^4$ from \cite{LSCW} fits in the present framework. 
        
        First, recall that in the $n=4$ case we have
        $E\cong \w{2} V_5$, $N\cong\w{4} V_5$, for $V_5:=\R^5$,
        with the maps $S^2E\to N$ and $S^2 E^*\to N^*$ given by wedging. Every co-Lagrangian of codimension 4 is of the form $\w{2} V_4$ for some 4-dimensional subspace $V_4\subset V_5$.
        
        A natural example is thus given by the Lie algebra case
        \[E:=\mf{so}(5),\qquad V:=\mf{so}(4),\]
        which produces a Leibniz parallelisable M-exact elgebroid over $S^4\cong \ms{SO}(5)/\ms{SO}(4)$.
%    \subsection{5-sphere}
%        Another example described in \cite{LSCW} is $S^5$, corresponding to the $n=6$ case. In contrast to the previous examples, the base is 5-dimensional, i.e.\ $V\subset E$ is of codimension $n-1$. 
%        
%        Motivated by the decomposition $\mathbf{27}=(\mathbf{15},\mathbf{1})\oplus (\mathbf{6},\mathbf{2})$ of $E$ under $\mf{sl}(6,\R)\oplus\mf{sl}(2,\R)\subset \mf e_{6(6)}$, we take
%        \[E:=\mf{so}(6)\oplus (V_6\oplus V_6),\qquad V:=\mf{so}(5)\oplus (V_6\oplus V_6),\]
%        where $V_6$ is the vector representation of $\mf{so}(6)$. The bracket on $E$ is defined as follows: for $x\in\mf{so}(6)$, $[x,\cdot]$ coincides with the $\mf{so}(6)$ representation on $E$, while $[V_6\oplus V_6,E]=0$. 
%        Note that this implies $\mc D\neq 0$ and $\ms G_E=SO(6)$, $\ms G_V=SO(5)$. 
        
%        This gives a $\ms G$-parallelisable exact elgebroid $E$ over $M=S^5\cong SO(6)/SO(5)$. As discussed in \cite{LSCW}, we have $E\cong TM\oplus (T^*M\oplus T^*M)\oplus\w{3}T^*M\oplus (\w{5}T^*M\oplus\w{5}T^*M)$. 
\section{Elgebroids and supergravity}\label{sec:sugra}
We now turn to applying the elgebroid framework to the study of supergravities given by a restriction of the 11-dimensional supergravity to lower dimensions, following \cite{CSCW}.
\subsection{Connections and torsion}
    \begin{defn}
        Let $E\to M$ be a $\ms G$-algebroid.
        A \emph{generalised connection} on $E$ is a map 
        \[\nabla\colon \Gamma(E)\times \Gamma(E)\to\Gamma(E),\quad (u,v)\mapsto \nabla_{\!u} v,\]
        satisfying
        \[\nabla_{\!fu}v=f\nabla_{\!u} v,\qquad \nabla_{\!u} fv=f\nabla_{\!u} v+(\rho(u)f)v,\]
        and such that $\nabla_{\!u}$ preserves the $\ms G$-structure for every $u\in\Gamma(E)$.
    \end{defn}
    \begin{defn}
        The \emph{torsion} of a generalised connection on $E$ is the map 
        \[T_\nabla\colon \Gamma(E)\times \Gamma(E)\to \Gamma(E),\quad T_\nabla(u,v)=\nabla_{\!u} v-\nabla_{\!v} u-[u,v]+((\nabla u\otimes v)_N)_E,\]
        where we understand $\nabla u$ as a section of $E^*\otimes E$.
    \end{defn}
    \begin{prop}
        Torsion is $C^\infty(M)$-bilinear, i.e.\ $T_\nabla\in\Gamma(E^*\otimes E^*\otimes E)$.
    \end{prop}
    \begin{proof}
        Follows immediately from \eqref{eq:anchor} and Part \eqref{eq:first_slot} of Lemma \ref{lem:basic}.
    \end{proof}
\subsection{Generalised metric and torsion-free compatible connections}
    
    Let us now specialise to the exceptional case. We start by recalling the construction from \cite{CSCW}. First, let $\ms K$ be the double cover of the maximal compact subgroup of $\ms G$, see Appendix \ref{sec:table}.
    
    \begin{defn}
        A \emph{generalised metric} is a reduction of the $\ms G$-structure on $E$ to a $\ms K$-structure.
    \end{defn}
    Physically, a generalised metric on an M-exact elgebroid encodes the bosonic field content of the lower-dimensional supergravity.\footnote{A generalised metric strictly only defines a ${\ms K}/\Z_2$-structure. However, since in what follows we will want to consider the exceptional group analogues of spinor representations we here use the stronger $\ms K$-structure definition.} 
    \begin{defn} 
        A \emph{compatible connection} is a generalised connection preserving the generalised metric.
    \end{defn}
    \begin{defn}
        A generalised metric is called \emph{torsion-free} if it admits a torsion-free compatible connection, i.e.\ if there is a compatible connection $\nabla$ with $T_\nabla=0$.
    \end{defn}
    The space of compatible connections is affine, over $\Gamma(E^*\otimes\on{ad}(\ms K))$, where $\on{ad}(\ms K)$ is the adjoint bundle corresponding to $\ms K$. Consider the vector bundle map
    \begin{align*}
        \lambda\colon E^*\otimes\on{ad}(\ms K)\to E^*\otimes E^*\otimes E,\qquad a\mapsto T_{\nabla+a}-T_\nabla,
    \end{align*}
    where $\nabla$ is a compatible connection ($\lambda(a)$ is independent of the choice of $\nabla$). %The kernel of this map has been worked out and corresponds to the representations $U$ of the group $K$ which can be found in Appendix \ref{sec:table}. 
    If a generalised metric is torsion-free, torsion-free compatible connections form an affine space over $\Gamma(\on{Ker}\lambda)$.
    Finally note that if $\nabla$ is a compatible connection then $\nabla_{\!u}$ acts also on any vector bundle associated to the generalised metric.
    
    Suppose now that we have a torsion-free generalised metric on an elgebroid $E$ and let $X$ be a bundle associated to some representation of $\ms K$. The action of $\on{ad}(\ms K)$ on $X$ induces the map
    \[\lambda_X\colon \on{Ker} \lambda\otimes X\to E^*\otimes \on{ad}(\ms K)\otimes X\to E^*\otimes X,\]
    which in turn gives us the projection $\mc P_X\colon E^*\otimes X\to (E^*\otimes X)/\on{Im}\lambda_X$.
    By construction we then have that 
    \[\mc P_X\circ \nabla\colon \Gamma(X)\to \Gamma((E^*\otimes X)/\on{Im}\lambda_X)\]
    is independent of the choice of the torsion-free connection $\nabla$.
\subsection{Curvature}
    For every $n\in\{4,\dots,6\}$ there are two important representations, labelled $S$ and $J$ in Appendix \ref{sec:table}, known as the \emph{spinor} and \emph{gravitino} representations, respectively.\footnote{For simplicity we are here excluding the $n=3$ case, which is somewhat simpler but does not respect the following pattern, on account of $\mc P_X$ being the identity for all $X$.} Notably, we have that the codomain of both $\mc P_S$ and $\mc P_J$ can be identified with $S\oplus J$. We thus have a map
    \[\mc P:=\mc P_S+\mc P_J\colon E^*\otimes(S\oplus J)\to S\oplus J.\]
    
    \begin{defn}
        Let $E$ be an elgebroid, with a torsion-free generalised metric and a compatible torsion-free connection $\nabla$. The \emph{generalised Ricci curvature} is the map 
        \[\mc R\colon \Gamma(S)\to\Gamma(S\oplus J),\qquad \epsilon\mapsto \mc P\circ \nabla\circ \mc P\circ\nabla \epsilon.\]
    \end{defn}
    It is immediately apparent, from the discussion above, that $\mc R$ is independent of the choice of $\nabla$.
    
    The following is proven in \cite{CSCW}. (We are here also using Theorem \ref{thm:class}.)
    \begin{prop}
        Suppose $E$ is an M-exact elgebroid with a generalised metric. Then the torsion of any compatible connection is in $\on{Im}\lambda$. Thus all generalised metrics on M-exact elgebroids are torsion-free.
    \end{prop}
    \begin{prop}
        For any generalised metric on an M-exact elgebroid, the Ricci curvature is a tensor, i.e.\ $\mc R\in\Gamma(S^*\otimes (S\oplus J))$.
    \end{prop}
    On an M-exact elgebroid the vanishing of this tensor corresponds to the equations of motion of the corresponding supergravity on $M$. A solution to these equations gives rise, after taking a product of $M$ with a Minkowski space, to a solution to the equations of 11-dimensional supergravity.
    
\subsection{Compatibility of Poisson--Lie U-duality and supergravity}
    Note that generalised metrics can be always pulled back, via pullbacks of elgebroids.
    \begin{thm}
        Suppose $E'\to M'$ is an M-exact pullback elgebroid of some $E\to M$ along a surjective submersion. Suppose that there is a generalised metric on $E$, inducing a generalised metric on $E'$. Then the generalised Ricci tensor on $E'$ vanishes iff the generalised Ricci tensor on $E$ vanishes.
    \end{thm}
    \begin{proof}
        Let $\varphi$ be the map $M' \to M$. First, let us show that the generalised metric on $E$ is torsion-free. Picking any compatible connection $\nabla$ on $E$, we have an induced compatible connection $\varphi^*\nabla$ on $E'$, defined by $(\varphi^*\nabla)_{\varphi^*u}\varphi^*v=\varphi^*(\nabla_{\!u}v)$. We then have $\varphi^*T_\nabla=T_{\varphi^*\nabla}\in \Gamma(\on{Im}\lambda')$, implying $T_\nabla\in\Gamma(\on{Im}\lambda)$. We can thus find another compatible connection on $E$, which is torsion-free. In particular, $\mc R$ on $E$ is well defined.
        The theorem then follows from the fact that $\mc R$ on $E'$ vanishes iff it vanishes on $\varphi^*u$, for all $u\in\Gamma(E)$. But $\mc R(\varphi^*u)=\varphi^*(\mc R u)$, which concludes the proof.
    \end{proof}
    This leads to the following consequences. First, having two different M-exact pullbacks (of the same $E$) on $E'$ and $E''$, with the generalised metrics induced by the one on $E$, the Ricci tensor vanishes on $E'$ iff it vanishes on $E''$. In other words, Poisson--Lie U-duality (in the M-theory case) is compatible with the equations of supergravity.
    
    Furthermore, let $E'$ be an M-exact pullback of an elgebra $E$. Solving $\mc R=0$ on $E$ is ``easy'', since the equation is purely algebraic. However, finding a solution and pulling it back to $E'$ produces an honest solution to the supergravity equations of motion.

{\appendix
\section{Exceptional groups and elgebroids}\label{sec:alg}
\subsection{List of exceptional groups and related data}\label{sec:table}
    We here provide a list of groups and representations relevant for the exceptional geometry, namely the split real forms of the ``exceptional'' groups and the double-cover $\ms K$ of their maximal compact subgroups, representations $E$ and $N$ of the exceptional group, and finally the representations $S$ and $J$ of $\ms K$ (the spinor and gravitino representations).
    \begin{center}
    \begin{tabular}{||c||c|c|c|c||}
        \hline
        $n$ & $3$ & $4$ & $5$ & $6$\\\hline\hline
        $\ms E_{n(n)}$ & $\ms{SL}(3,\R)\times \ms{SL}(2,\R)$ & $\ms{SL}(5,\R)$ & $\ms{Spin}(5,5)$ & $\ms E_{6(6)}$\\\hline
        $\ms K$ & $(\ms{Spin}(3)\times \ms{Spin}(2))/\Z_2$ & $\ms{Spin}(5)$ & $\ms{Spin}(5)\times \ms{Spin}(5)$ & $\ms{USp}(8)$\\\hline\hline
        $E$ & $\mathbf{(3,2)}$ & $\mathbf{10}$ & $\mathbf{16}$ & $\mathbf{27}$\\\hline
        $N$ & $\mathbf{(3',1)}$ & $\mathbf{5'}$ & $\mathbf{10}$ & $\mathbf{27'}$\\\hline\hline
        $S$ & $\mathbf{2_1}\oplus\mathbf{2_{-1}}$ & $\mathbf{4}$ & $\mathbf{(4,1)\oplus (1,4)}$ & $\mathbf{8}$\\\hline
        $J$ & $\mathbf{4_1}\oplus\mathbf{4_{-1}}\oplus\mathbf{2_3}\oplus\mathbf{2_{-3}}$ & $\mathbf{16}$ & $\mathbf{(4,5)\oplus(5,4)}$ & $\mathbf{48}$\\\hline
        %$U$ & $\mathbf{0}$ & $\mathbf{35}$ & $\mathbf{(20,4)\oplus (4,20)}$ & $\mathbf{594}$\\\hline
    \end{tabular}
    \end{center}

    \subsection{Algebra}\label{subsec:alg_gld}
    Let us be more explicit about the Lie algebra $\mf{e}_{n(n)}\oplus \R$. In terms of its $\mf{gl}(T)$-subalgebra, for $T:=\R^n$, it decomposes as
    \[\mf e_{n(n)}\oplus \R=\R\oplus\mf{gl}(T)\oplus \w{3} T^*\oplus \w{6} T^*\oplus \w{3} T\oplus \w{6} T.\]
    First, the $\R$ factor is central. Writing 
    \[a_3+a_6+w_3+w_6\in \w{3} T^*\oplus \w{6} T^*\oplus \w{3} T\oplus \w{6} T,\]
    the remaining nontrivial brackets are
    \begin{align*}
        [a_3,a'_3]=-a_3\wedge a_3'&,\quad [w_3,w_3']=-w_3\wedge w_3',\quad [a_6,w_3]=\iota_{w_3}a_6,\quad [a_3,w_6]=\iota_{w_6}a_3,\\
        [w_3,a_3]&=\hphantom{-}(a_3\star w_3-\tfrac13 \la a_3,w_3\ra \mathds{1})+\tfrac13 \la a_3,w_3\ra\in\mf{gl}(T)\oplus\R\\
        [w_6,a_6]&=-(a_6\star w_6-\tfrac23 \la a_6,w_6\ra \mathds{1})-\tfrac23 \la a_6,w_6\ra\in\mf{gl}(T)\oplus\R,
    \end{align*}
    with
    \[\star\colon \w{k} \!T^*\otimes\w{k} T \to\mf{gl}(T)\cong\on{Hom}(T\otimes T^*,\R),\qquad \alpha\star w=\la\iota_{\bullet}\alpha,\iota_\bullet w\ra,\]
    The algebra $\mf{e}_{n(n)}$ is embedded by setting the $\R$-component equal to the trace of the $\mf{gl}(T)$ component divided by $9-n$.
    
    The representation $E$ is given as follows. First, the action of $\mf{gl}(T)$ is given by the decomposition
    \begin{equation}\label{eq:fund_rep}
        E=T\oplus \w{2}T^*\oplus \w{5}T^*,
    \end{equation}
    while $\R$ acts with weight $1$. Writing $u=X+\sigma_2+\sigma_5\in E$, the remaining part is given by
    \begin{align*}
        w_3\cdot u=\iota_{w_3}(\sigma_2+\sigma_5),\;\; w_6\cdot u=-\iota_{w_6}\sigma_5,\;\; a_3\cdot u=\iota_X a_3+a_3\wedge\sigma_2, \;\; a_6 \cdot u=\iota_X a_6.
    \end{align*}
    
    \subsection{Classification of Lagrangian and co-Lagrangian subspaces}\label{ap:class}
    Recall that the formula for $S^2E\to N$ was given in Example \ref{ex:exc_data}.
    \begin{lem}
         Let $n\in\{3,\dots,6\}$. Consider $u\in E=T\oplus \w{2}T^*\oplus \w{5}T^*$. If $(u\otimes u)_N=0$ then there exists $g\in \ms E_{n(n)}\times\R^+$ s.t.\ $g\cdot u\in T$. If furthermore $u$ has a non-vanishing $T$-part, this can be achieved via an element of the nilpotent subgroup $\w{3}T^*\oplus \w{6}T^*\subset \ms E_{n(n)}\times\R^+$.
    \end{lem}
    \begin{proof}
        Suppose first that $u=X+\sigma_2+\sigma_5$, with $X\neq 0$. Let $\xi$ be an element of $T^*$ satisfying $\la \xi,X\ra=1$. Since $(u\otimes u)_N=0$ implies $i_X\sigma_2=0$, taking $a_3=-\sigma_2\wedge \xi$ we have that $e^{a_3}\cdot u=X+\sigma_5'$, with $\sigma_5'\in \w{5}T^*$. Continuing, taking $a_6=\sigma_5'\wedge \xi$, we get $e^{a_6}\cdot (X+\sigma_5')=X$.
        
        We are left to show that if $u=\sigma_2+\sigma_5$, then there exists $g$ s.t.\ $g\cdot u$ will have a nonzero $T$-part. Suppose $\sigma_2\neq 0$. Note that $(u\otimes u)_N=0$ implies $\sigma_2\wedge\sigma_2=0$, i.e.\ $\sigma_2$ is decomposable\footnote{By definition, a 2-form $\sigma_2$ is decomposable if it can be written as a wedge product of two 1-forms, or equivalently if $\sigma_2\wedge\sigma_2=0$.}. Let $o\in \w{2}T$ be a decomposable element s.t.\ $\la \sigma_2,o\ra=1$ and let $Y\in T\neq 0$ be such that $i_Y\sigma_2=0$. Setting $w_3:=o\wedge Y$, we have $e^{w_3}\cdot (\sigma_2+\sigma_5)=Y+(\sigma_2+i_{w_3}\sigma_5)+\sigma_5$, since $w_3\cdot i_{w_3}\sigma_5=0$.
        
        Finally, if $u=\sigma_5\neq 0$, taking any $w_3$ s.t.\ $i_{w_3}\sigma_5\neq 0$ will produce a $2$-form part in $e^{w_3}\cdot\sigma_5$, yielding the previous case.
    \end{proof}
    \begin{prop}\label{prop:classcoiso}
        The space of Lagrangian subspaces of $E$ consists of 2 orbits of the action of $\ms E_{n(n)}\times\R^+$, given by $n$ and $n-1$-dimensional subspaces, respectively. 
    \end{prop}
    \begin{proof}
        We shall show that, up to an $\ms E_{n(n)}\times \R^+$-transformation we obtain but two possibilities. We proceed inductively. 
        
        Suppose that such a Lagrangian subspace $W$ is spanned by vectors $\omega_i$, $i\in\{1,\dots\}$. Since $(\omega_1\otimes\omega_1)_N=0$, we can find $g$ such that $g\cdot \omega_1\in T$ -- we now replace $W$ by $g\cdot W$, i.e.\ all $\omega_i$ by $g\cdot \omega_i$. 
        
        Let now $U\subset T$ be an $n-1$-dimensional subspace of $T$ which is complementary to $\la \omega_1\ra\in T$, where $\la\cdot\ra$ denotes the linear span. The remaining $\omega_i$'s satisfy $(\omega_1\otimes\omega_i)_N=0$, implying they belong to the subspace
        $U\oplus \w{2}U^*\oplus \w{5}U^*\oplus \la \omega_1\ra$. Replacing $\omega_i$, $i\in\{2,\dots\}$ by $\omega_i+\lambda_i \omega_1$, for some suitable $\lambda_i$, we get that
        \[\omega_i\in U\oplus \w{2}U^*\oplus \w{5}U^*,\quad i\in\{2,\dots\}.\]
        Similarly, the Lie subalgebra of $\mf e_{n(n)}\oplus \R$ preserving the subspace spanned by $\omega_1$ contains the algebra $\mf e_{n-1(n-1)}\oplus \R\cong \R\oplus\mf{gl}(U)\oplus \w{3} U^*\oplus \w{6} U^*\oplus \w{3} U\oplus \w{6} U$. Thus the problem for a given $n$ reduces to the same problem for $n-1$.
        
        To finish, we only need to look at the case $n=2$,\footnote{Although above we only considered the case of $n\ge 3$, both the $\mf{gl}(n,\R)$ decompositions of $E$, $N$, as well as the form of the map $E\otimes E\to N$, extend naturally to the $n=2$ case.} where we have $E\cong T\oplus \w{2}T^*$, with $\dim T=2$, $N\cong T^*$, and $((X+\sigma_2)\otimes (X+\sigma_2))_N=2i_X\sigma_2$. If $\omega_1$ has a non-zero 2-form part, the condition $(\omega_1\otimes\omega_1)_N=0$ requires it to have a vanishing vector part. This gives the 1-dimensional Lagrangian subspace $\w{2}T^*$. The other possibility is to have $\omega_1\in T$, which can be further enlarged by adding a second generator of $T$, yielding the 2-dimensional Lagrangian subspace $T$.
    \end{proof}
    \begin{prop}\label{prop:pairs}
        All pairs $(V,W)$, given by a co-Lagrangian subspace $V\subset E$ of codimension $n$ and a complementary Lagrangian subspace $W$, are related by the action of the group $\ms E_{n(n)}\times\R^+$.
    \end{prop}
    \begin{proof}
        Since the formulas for $S^2 E^*\to N^*$, in terms of the $\mf{gl}(T)$-decomposition, have (up to an overall constant) the same form as the ones for $S^2E\to N$, we get that up to the $\ms E_{n(n)}\times\R^+$-action there is just one codimension $n$ co-Lagrangian subspace $V\subset E$. Let us therefore identify $V$ with $\w{2}T^*\oplus\w{5}T^*\subset E$. Note that this is preserved by the subgroup of $\ms E_{n(n)}\times\R^+$ corresponding to $\R\oplus\mf{gl}(T)\oplus \w{3}T^*\oplus \w{6}T^*\subset \mf{e}_{n(n)}\oplus\R$.
        
        Suppose again that $W$ is spanned by $\omega_i$. Since $\omega_i\notin V$, the Lemma implies that we can use a $\w{3}T^*\oplus\w{6}T^*$-transformation to map $\omega_1$ into an element of $T$. Redefining the basis of $W$, we can assume that the remaining $\omega_i$'s lie in $U\oplus \w{2}U^*\oplus \w{5}U^*$, where $U$ is a complement to $\la\omega_1\ra\subset T$, and the situation reduces to the same situation in a smaller dimension. Ultimately, we reach $n=2$, in which the only possible 2-dimensional Lagrangian complementary to $\w{2}T^*$ is $W=T$.
    \end{proof}
    
    \subsection{Rewriting the bracket}
    For the purpose of the proof, it will be useful to recast the bracket on the exceptional tangent bundle in a more convenient language. Following \cite{CSCW}, this is given as follows.
    
    First, pick local coordinates on $M$. This (locally) induces a trivialisation $E\cong M\times (T\oplus\w{2}T^*\oplus \w{5}T^*)$, with $T:=\R^n$, and thus
    $\Gamma(E)\cong C^\infty(M)\otimes (T\oplus \w{2} T^*\oplus \w{5} T^*)$ and similarly for $N$.
    We then have
    \begin{equation}\label{eq:bracket_coords}
        [u,v]=\rho(u)v-\pi(\hat d u)v.
    \end{equation}
    One can check that this is independent of the choice of coordinates.

    \subsection{Pre-elgebroids}
    \begin{defn}
        A \emph{pre-elgebroid} is a structure obtained by replacing, in the definition of an elgebroid, the condition \eqref{eq:jacobi} by a weaker condition \eqref{eq:morphism} from Lemma \ref{lem:basic}. A pre-elgebroid is \emph{M-exact} if the sequence $T^*M\otimes N\to E\to TM\to 0$ is exact and $\dim M=n$.
    \end{defn}
    Note that in particular the properties \eqref{eq:coiso} and \eqref{eq:first_slot} from Lemma \ref{lem:basic} still hold for a pre-elgebroid.
    \begin{lem}\label{lem:large}
         An M-exact pre-elgebroid is locally of the form from Example \ref{ex:exc}, but with the bracket
         \begin{align}\label{eq:long}
            [X+\sigma_2+\sigma_5&,X'+\sigma_2'+\sigma_5']=\mc L_X X'+(\mc L_X\sigma_2'-i_{X'}d\sigma_2+i_{X'}i_X F_4+(i_XF_1)\sigma_2'-i_{X'}(F_1\wedge\sigma_2))\nonumber\\
            &+(\mc L_X \sigma_5'-i_{X'}d\sigma_5-\sigma_2'\wedge d\sigma_2+(i_XF_4)\wedge\sigma_2'-i_{X'}(F_4\wedge \sigma_2)+2(i_XF_1)\sigma_5'\\
            &-F_1\wedge\sigma_2\wedge\sigma_2'-2i_{X'}(F_1\wedge \sigma_5)),\nonumber
        \end{align}
        for some $F_1\in\Omega^1(M)$, $F_4\in\Omega^4(M)$.
    \end{lem}
    \begin{proof}
        By Lemma \ref{lem:colagexc}, M-exactness implies that $\rho$ is surjective and $\on{Ker}\rho$ is co-Lagrangian of codimension $n$. Choose a local isotropic splitting $\iota\colon TM\to E$ of the exact sequence. Since the base $M$ is $n$-dimensional, $\iota(TM)$ is automatically Lagrangian, and we have a decomposition $E=\on{Ker}\rho\oplus \iota(TM)$ into a codimension $n$ co-Lagrangian subbundle and a Lagrangian one. Using Proposition \ref{prop:pairs}, we can then make an identification 
        \begin{equation}\label{eq:identif}
            E=\iota(TM)\oplus\on{Ker}\rho\cong TM\oplus(\w{2} T^*M\oplus \w{5} T^*M),
        \end{equation}
        and similarly $N\cong T^*M\oplus \w{4}T^*M$, with the maps between $S^2E$ and $N$ (and the action of $\ms E_{n(n)}\times\R^+$) given as in the case of the exceptional tangent bundle. This identification is not unique, due to the presence of the $\R^+$-factor in the group $\R^+\times \ms{GL}(T)$ preserving the decomposition \eqref{eq:fund_rep}. We can however always make the choice locally, with two such choices differing by a positive function $e^\psi$ for $\psi\in C^\infty(M)$.
        
        It remains to check that the bracket has the desired form. Picking local coordinates on $M$, we get a trivialisation of $E$ just as in the previous Subsection. From \eqref{eq:anchor}, part \eqref{eq:first_slot} of Lemma \ref{lem:basic}, and the fact that $[u,\cdot]$ preserves the $\ms E_{n(n)}\times\R^+$-structure, we get
        \[[u,v]=\rho(u)v-\pi(\hat d u)v+A(u)\cdot v,\]
        where $A$ is (at every point of $M$) a map 
        \[T\oplus \w{2}T^*\oplus \w{5}T^*\to\R\oplus\w{6}T\oplus\w{3} T\oplus \mf{gl}(T)\oplus\w{3}T^*\oplus \w{6}T^*.\]
        Similarly, we have $\mc Dn=(\hat dn)_E+B(n)$, with 
        \[B\colon T^*\oplus \w{4} T^* \oplus (T^*\otimes\w{6}T^*)\to T\oplus \w{2}T^*\oplus \w{5}T^*.\]
        
        Taking two constant sections $u,v$, we have $[u,v]=A(u)\cdot v$ and also $\rho([u,v])=0$, implying \[A(u)(T\oplus\w{2}T^*\oplus\w{5}T^*)\subset \w{2}T^*\oplus\w{5}T^*.\] Thus $A$ is actually targeted only in $\R'\oplus\w{3}T^*\oplus \w{6}T^*$, where 
        \[\R'\subset \mf e_{n(n)}\oplus\R,\qquad \R'=\{(\tfrac{c}{3},-\tfrac{c}{3}\mathds{1})\in \R\oplus \mf{gl}(T)\mid c\in\R\}.\] In particular, $\R'$ acts on $T$, $\w{2}T^*$, and $\w{5}T^*$ with weights $0$, $1$, and $2$, respectively. Let us use the notation $A_0$, $A_3$, $A_6$ for the parts of $A$ valued in $\R'$, $\w{3}T^*$, $\w{6}T^*$.
        
        Since $(T\otimes T)_N=0$, using \eqref{eq:sym} we have that for $X,Y\in T$ 
        \[0=A(X)\cdot Y+A(Y)\cdot X=i_Y (A_3(X)+A_6(X))+i_X (A_3(Y)+A_6(Y)).\]
        This implies that $A|_T$, seen as an element of $(T^*\otimes\w{0}T^*)\oplus (T^*\otimes \w{3}T^*)\oplus (T^*\otimes \w{6}T^*)$ is skew-symmetric in each of its terms, implying there exist $F_1\in T^*$ and $F_4\in \w{4} T^*$ s.t.\ $A(X)=i_X (F_1+F_4)$. Similarly, denoting 2- and 5-forms by the corresponding subscript, we have
        \[B(i_X\sigma_2)=A(X)\cdot \sigma_2+A(\sigma_2)\cdot X=(i_X (F_1+F_4))\wedge \sigma_2+i_X A_3(\sigma_2)+i_X A_6(\sigma_2).\]
        
        Thus $i_X \sigma_2=0$ implies $i_X [(F_1+F_4)\wedge \sigma_2+A_3(\sigma_2)+A_6(\sigma_2)]=0$. Taking $\sigma_2$ decomposable, there are $n-2$ independent vectors in $T$ which give zero upon contraction with $F_1\wedge\sigma_2+A_3(\sigma_2)\in\w{3}T^*$ and $F_4\wedge\sigma_2+A_6(\sigma_2)\in\w{6}T^*$. This implies 
        \[A_3(\sigma_2)=-F_1\wedge\sigma_2,\qquad A_6(\sigma_2)=-F_4\wedge\sigma_2.\]
        Furthermore, for $\sigma_2$ decomposable,
        $0=-B(\sigma_2\wedge\sigma_2)=2A(\sigma_2)\cdot \sigma_2=2A_0(\sigma_2)\sigma_2$, implying $A_0(\sigma_2)=0$.
        Since decomposable 2-forms span $\w{2} T^*$, we get
        \[A(\sigma_2)=-F_1\wedge \sigma_2-F_4\wedge \sigma_2 \quad \forall \sigma_2\in\w{2}T^*.\]
        
        Next, from
        $-B(\sigma_2\wedge\sigma_2')=A(\sigma_2)\cdot\sigma_2'+A(\sigma_2')\cdot\sigma_2=-2F_1\wedge\sigma_2\wedge\sigma_2'$
        we deduce $B(\sigma_4)=2F_1\wedge\sigma_4$. This in turn gives
        \[2F_1\wedge i_X\sigma_5=B(i_X \sigma_5)=A(X)\cdot\sigma_5+A(\sigma_5)\cdot X=2(i_XF_1)\sigma_5+A(\sigma_5)\cdot X,\]
        implying $A_3(\sigma_5)=0$ and $A_6(\sigma_5)=-2F_1\wedge\sigma_5$. For any $\sigma_5$ there exists $\sigma_2\neq 0$ such that $j\sigma_2\wedge\sigma_5=0$, implying
        \[0=B(j\sigma_2\wedge \sigma_5)=A(\sigma_2)\cdot\sigma_5+A(\sigma_5)\cdot \sigma_2=A_0(\sigma_5)\sigma_2,\]
        and thus $A_0(\sigma_5)=0$. Putting things together and using \eqref{eq:bracket}, we obtain bracket of the desired form.
    \end{proof}
    \begin{lem}\label{lem:jacobiator}
         For any $M$-exact pre-elgebroid, the \emph{Jacobiator}
         \[J(u,v,w):=[u,[v,w]]-[[u,v],w]-[v,[u,w]]\]
         is $C^\infty(M)$-linear in all the slots.
    \end{lem}
    \begin{proof}
        The claim follows from a straightforward calculation using formula \eqref{eq:long}.
    \end{proof}
    This immediately implies:
    \begin{cor}\label{cor:preelgtoelg}
        If an M-exact pre-elgebroid locally admits a trivialisation such that the Jacobiator of constant sections vanishes, then it is an exact elgebroid.
    \end{cor}
    \subsection{Proof of Theorem \ref{thm:class}}\label{ap:proof_class}
    \begin{proof}
        Applying Lemma \ref{lem:large}, we locally get a bracket of the form \eqref{eq:long}. A quick calculation then reveals
        \[[X,[Y,\sigma_2]]-[[X,Y],\sigma_2]-[Y,[X,\sigma_2]]=\sigma_2\wedge i_Yi_X(dF_4+F_1\wedge F_4+dF_1).\]
        Thus axiom \eqref{eq:jacobi} from the definition of an elgebroid requires $dF_1=dF_4+F_1\wedge F_4=0$. Conversely, it is straightforward to check that for any $F_1$ and $F_4$ satisfying these conditions, the axiom is satisfied for all $u,v,w\in\Gamma(E)$.
        
        Taking a different choice of the identification \eqref{eq:identif}, we have
        \[F_1\to F_1+d\psi,\qquad F_4\to e^{-\psi} F_4.\]
        We can therefore locally always achieve $F_1=0$ and $dF_4=0$.
        Finally, note that at a point $p\in M$ any other Lagrangian splitting $TM\to E$ is related to our chosen one via the action of an element from the nilpotent subgroup $\w{3}T^*_pM\oplus \w{6}T^*_pM$ of $\ms E_{n(n)}\times\R^+$. Assuming $F_1=0$, changing the splitting by an element $A_3+A_6\in \Omega^3(M)\oplus\Omega^6(M)$ modifies the bracket by $F_4\to F_4+dA_3$, which means that we can always locally find a splitting such that the bracket has the form \eqref{eq:bracket} with $F_1=F_4=0$.
    \end{proof}
    
    \subsection{Proof of Theorem \ref{thm:action}}\label{ap:proof_action}
    \begin{proof}
    In general, an elgebroid is M-exact iff it is transitive and $\on{Ker}\rho$ is at every point (on $M'$) co-Lagrangian and of codimension $n$. Therefore, if the pullback is to be M-exact, the stabilisers of the action must be co-Lagrangian and of codimension $n$. We will now show that this is the only requirement.
    
    Let us make the identification $\Gamma(E')\cong C^\infty(M')\otimes E$ and similarly for $E'^*$ and $N'$. Equations \eqref{eq:anchor}, \eqref{eq:diff}, and part \eqref{eq:first_slot} of Lemma \ref{lem:basic}, imply that $[\cdot,\cdot]'$ and $\mc D'$ necessarily take the form
    \[[u,v]'=[u,v]+\chi(u)v-\pi(\hat d u)v,\qquad \mc D'n=\mc Dn+(\hat dn)_E,\]
    where $\dh f=\chi^t(df)\in C^\infty(M')\otimes E^*$ for any $f\in C^\infty(M')$.
    One easily verifies that this satisfies \eqref{eq:anchor}, \eqref{eq:sym} and \eqref{eq:diff}, and the bracket $[u,\cdot]$ preserves the $\ms E_{d(d)}\times\R^+$-structure (in the last condition we use Definition \ref{def:grpdata}).
    Finally, for any (i.e.\ not necessarily constant) sections $u,v\in \Gamma(E')$ we have
    \[[\rho'(u),\rho'(v)]'-\rho'([u,v]')=-\chi(((\hat d u\otimes v)_N)_E)=0,\]
    due to the coisotropy.
    Thus $E'$ is a pre-elgebroid. Since the Jacobiator of constant sections coincides with the vanishing Jacobiator on $E$, we can use Corollary \ref{cor:preelgtoelg} to conclude the proof.
\end{proof}
    }
        
\end{document}